\documentclass[10pt,reqno]{amsart}

\usepackage[dvipsnames]{xcolor}
\usepackage{amsmath,amssymb,amsthm,amsfonts,enumerate,tikz,bm}
\usepackage{mathrsfs}
\usetikzlibrary{matrix,arrows,positioning,calc}
\usepackage[all]{xy}
\usepackage[bookmarksnumbered,colorlinks]{hyperref}
\usepackage{url}
\numberwithin{equation}{section}
\usepackage{graphicx}
\usepackage[T1]{fontenc}
\usepackage{textcomp}
\usepackage{palatino,helvet}

\usepackage{footmisc}
\usepackage{float}
\usepackage{caption}

\newtheorem{definition}{Definition}[section]
\newtheorem{remark}[definition]{Remark}
\newtheorem{example}[definition]{Example}
\newtheorem{theorem}[definition]{Theorem}
\newtheorem{proposition}[definition]{Proposition}
\newtheorem{lemma}[definition]{Lemma}
\newtheorem{corollary}[definition]{Corollary}
\theoremstyle{remark}
\numberwithin{equation}{section}

\usepackage{paralist}

\usepackage{scrextend}
\deffootnote{2em}{0em}{\thefootnotemark\quad}

\newcommand{\mcA}{\mathcal{A}}
\newcommand{\mcB}{\mathcal{B}}
\newcommand{\mcC}{\mathcal{C}}

\newcommand{\mcF}{\mathcal{F}}

\newcommand{\mcI}{\mathcal{I}}
\newcommand{\mcL}{\mathcal{L}}
\newcommand{\mcP}{\mathcal{P}}

\newcommand{\mcGP}{\mathcal{GP}}

\newcommand{\mcGI}{\mathcal{GI}}

\newcommand{\Mod}{\mathsf{Mod}}

\newcommand{\resdim}{{\rm resdim}}
\newcommand{\coresdim}{{\rm coresdim}}
\newcommand{\pd}{{\rm pd}}
\newcommand{\id}{{\rm id}}
\newcommand{\Gpd}{{\rm Gpd}}
\newcommand{\Gid}{{\rm Gid}}

\newcommand{\Hom}{{\rm Hom}}
\newcommand{\Ext}{{\rm Ext}}

\makeatletter
\def\@seccntformat#1{%
  \protect\textup{\protect\@secnumfont
    \ifnum\pdfstrcmp{section}{#1}=0 \scshape\bfseries\fi% section # in \scshape and \bfseries
    \ifnum\pdfstrcmp{subsection}{#1}=0 \bfseries\fi% subsection # in \bfseries
    \csname the#1\endcsname
    \protect\@secnumpunct
  }%
}
\makeatother

\begin{document}

\title{Corrigendum to: ``$m$-periodic Gorenstein objects'', [J. Algebra 621 (2023)]}

\author{Mindy Y. Huerta}
\address[M. Y. Huerta]{Facultad de Ciencias, Universidad Nacional Aut\'onoma de M\'exico. Circuito Exterior, Ciudad Universitaria. CP04510. Mexico City, MEXICO. 
}
\email{mindy@matem.unam.mx}

\author{Octavio Mendoza}
\address[O. Mendoza]{Instituto de Matem\'aticas. Universidad Nacional Aut\'onoma de M\'exico. Circuito Exterior, Ciudad Universitaria. CP04510. Mexico City, MEXICO}
\email{omendoza@matem.unam.mx}

\author{Marco A. P\'erez}
\address[M. A. P\'erez]{Instituto de Matem\'atica y Estad\'istica ``Prof. Ing. Rafael Laguardia''. Facultad de Ingenier\'ia. Universidad de la Rep\'ublica. CP11300. Montevideo, URUGUAY}
\email{mperez@fing.edu.uy}

\maketitle

\begin{abstract}
Let $(\mathcal{A,B})$ be a GP-admissible pair and $(\mathcal{Z,W})$ be a GI-admissible pair of classes of objects in an abelian category $\mathcal{C}$, and consider the class $\pi\mathcal{GP}_{(\omega,\mathcal{B},1)}$ of $1$-periodic $(\omega,\mathcal{B})$-Gorenstein projective objects, where $\omega := \mathcal{A} \cap \mathcal{B}$ and $\nu := \mathcal{Z} \cap \mathcal{W}$. We claimed in \cite[Lem. 8.1]{HMP2023m} that the $(\mathcal{Z,W})$-Gorenstein injective dimension of $\pi\mathcal{GP}_{(\omega,\mathcal{B},1)}$ is bounded by the $(\mathcal{Z,W})$-Gorenstein injective dimension of $\omega$, provided that: (1) $\omega$ is closed under direct summands, (2) $\Ext^1(\pi\mathcal{GP}_{(\omega,\mathcal{B},1)},\nu) = 0$, and (3) every object in $\pi\mathcal{GP}_{(\omega,\mathcal{B},1)}$ admits a $\Hom(-,\nu)$-acyclic $\nu$-coresolution. These conditions are their duals are part of what we called ``Setup 1''. Moreover, if we replace $\pi\mathcal{GP}_{(\omega,\mathcal{B},1)}$ by the class $\mathcal{GP}_{(\mathcal{A,B})}$ of $(\mathcal{A,B})$-Gorenstein projective objects, the resulting inequality is claimed to be true under a set of conditions named ``Setup 2''. 

The proof we gave for the claims $\Gid_{(\mathcal{Z,W})}(\pi\mathcal{GP}_{(\omega,\mathcal{B},1)}) \leq \Gid_{(\mathcal{Z,W})}(\omega)$ and $\Gid_{(\mathcal{Z,W})}(\mathcal{GP}_{(\mathcal{A,B})}) \leq \Gid_{(\mathcal{Z,W})}(\omega)$ is incorrect, and the purpose of this note is to exhibit a corrected proof of the first inequality, under the additional assumption that every object in $\pi\mathcal{GP}_{(\omega,\mathcal{B},1)}$ has finite injective dimension relative to $\mathcal{Z}$. Setup 2 is no longer required, and as a result the second inequality was removed. We also fix those results in \S \ 8 of \cite{HMP2023m} affected by Lemma 8.1, and comment some applications and examples. 
\end{abstract}

%\setcounter{tocdepth}{2}
%\tableofcontents

\pagestyle{myheadings}
\markboth{\rightline {\scriptsize M. Huerta, O. Mendoza and M. A. P\'{e}rez}}
         {\leftline{\scriptsize $m$-periodic Gorenstein objects}}

%%%%%%%%%%%%%%%%%%%%%%%%%%%%%%%%%%%%%
%%%%%%%%%%%%%%%%%%%%%%%%%%%%%%%%%%%%%
%%%%%%%%%%%%%%%%%%%%%%%%%%%%%%%%%%%%%
%%%%%%%%%%%%%%%%%%%%%%%%%%%%%%%%%%%%%

\section{The error}

Throughout, we maintain the notation used in \cite{HMP2023m}. So in what follows, given an abelian category $\mcC$, $(\mathcal{A,B})$ will be a GP-admissible pair and $(\mathcal{Z,W})$ a GI-admissible pair in $\mcC$. We shall write $\omega := \mcA \cap \mcB$ and $\nu := \mathcal{Z} \cap \mathcal{W}$ for simplicity.  We shall prove several results concerning relative Gorenstein dimensions within the following \textbf{relative setting}: 
\begin{enumerate}
\item[($\mathsf{rs1}$)] $\omega$ and $\nu$ are closed under direct summands.

\item[($\mathsf{rs2}$)] $\Ext^{1}(\pi\mathcal{GP}_{(\omega,\mcB,1)}, \nu) = 0$ and $\Ext^1(\omega,\pi\mathcal{GI}_{(\mathcal{Z},\nu,1)}) = 0$.

\item[($\mathsf{rs3}$)] Every object in $\pi\mathcal{GP}_{(\omega, \mcB, 1)}$ admits a $\Hom(-,\nu)$-acyclic $\nu$-coresolution, and every object in $\pi\mathcal{GI}_{(\mathcal{Z}, \nu, 1)}$ admits a $\Hom(\omega,-)$-acyclic $\omega$-resolution. 

\item[($\mathsf{rs4}$)] Every object in $\mathcal{GP}_{(\mathcal{A,B})}$ admits a $\omega$-resolution, and every object in $\mathcal{GI}_{(\mathcal{Z,W})}$ admits a $\nu$-coresolution. 

\item[($\mathsf{rs5}$)] $\nu$ (resp., $\omega$) is closed under arbitrary (co)products, in the case $\mcC$ is AB4${}^\ast$ (resp., AB4). 
\end{enumerate}

Given $C \in \mcC$, the \emph{$(\mcA,\mcB)$-Gorenstein projective dimension of $C$} \cite[Def. 3.3]{BMS}, denoted ${\rm Gpd}_{(\mcA,\mcB)}(C)$, is defined as the resolution dimension 
\[
{\rm Gpd}_{(\mcA,\mcB)}(C) = \resdim_{\mathcal{GP}_{(\mcA,\mcB)}}(C).
\]
Similarly, 
\[
{\rm Gid}_{(\mathcal{Z,W})}(C) = \coresdim_{\mathcal{GI}_{(\mathcal{Z,W})}}(C)
\]
denotes and defines the \emph{$(\mathcal{Z,W})$-Gorenstein injective dimension of $C$}. If $\mathcal{X}$ is a class of objects in $\mcC$, the \emph{$(\mcA,\mcB)$-Gorenstein projective dimension} and \emph{$(\mathcal{Z,W})$-Gorenstein injective dimension of $\mathcal{X}$} are defined as
\begin{align*}
{\rm Gpd}_{(\mcA,\mcB)}(\mathcal{X}) & := \resdim_{\mathcal{GP}_{(\mcA,\mcB)}}(\mathcal{X}), \\
{\rm Gid}_{(\mathcal{Z,W})}(\mathcal{X}) & := \coresdim_{\mathcal{GI}_{(\mathcal{Z,W})}}(\mathcal{X}).
\end{align*}
Below $\mathcal{I}_\mathcal{Z}^{<\infty}$ denotes the class of objects $M \in \mathcal{C}$ such that $\id_{\mathcal{Z}}(M) < \infty$. Dually, $\mathcal{P}_{\mcB}^{<\infty}$ denotes the class of objects in $\mathcal{C}$ with finite projective dimension relative to $\mathcal{B}$.

Lemma 8.1 in \cite{HMP2023m} is stated incorrectly, and it should be replaced by the following statement.

\begin{lemma}\label{lem: Gid(piXcapY)<=n} 
If conditions ($\mathsf{rs1}$), ($\mathsf{rs2}$) and ($\mathsf{rs3}$) are satisfied, then 
\begin{align*}
\Gid_{(\mathcal{Z, W})}(\pi\mathcal{GP}_{(\omega, \mcB, 1)} \cap \mathcal{I}_\mathcal{Z}^{<\infty}) & \leq \Gid_{(\mathcal{Z, W})}(\omega), \\ 
\Gpd_{(\mathcal{A,B})}(\pi\mathcal{GI}_{(\mathcal{Z}, \nu, 1)} \cap \mathcal{P}_{\mcB}^{<\infty}) & \leq \Gpd_{(\mathcal{A,B})}(\nu).
\end{align*}
\end{lemma}

The purpose of this note is to prove the previous lemma, and also to correct those results and applications in \cite[\S \ 8]{HMP2023m} affected by the incorrect statement.

%%%%%%%%%%%%%%%%%%%%%%%%%%%%%%%%%%%%%
%%%%%%%%%%%%%%%%%%%%%%%%%%%%%%%%%%%%%
%%%%%%%%%%%%%%%%%%%%%%%%%%%%%%%%%%%%%
%%%%%%%%%%%%%%%%%%%%%%%%%%%%%%%%%%%%%

\section{Corrected proof and outcomes}

\begin{proof}[Proof of Lemma \ref{lem: Gid(piXcapY)<=n}] 
We only prove the first inequality since the other one is dual. Let $M \in \pi\mathcal{GP}_{(\omega, \mcB, 1)} \cap \mathcal{I}_\mathcal{Z}^{<\infty}$, and without loss of generality let $n := \Gid_{(\mathcal{Z, W})}(\omega) < \infty$. We first consider a $\Hom(-,\mcB)$-acyclic and exact sequence $M \rightarrowtail W \twoheadrightarrow M$ with $W \in \omega$. This sequence is also $\Hom(-, \nu)$-acyclic by ($\mathsf{rs2}$). Now, from ($\mathsf{rs3}$) and the dual of \cite[Lem. 2.4]{BMPbalanced}, we can construct a solid diagram
\[
\parbox{2.5in}{
\begin{tikzpicture}[description/.style={fill=white,inner sep=2pt}] 
\matrix (m) [ampersand replacement=\&, matrix of math nodes, row sep=2.5em, column sep=2.5em, text height=1.25ex, text depth=0.25ex] 
{ 
M \& W \& M \\ 
V^0 \& V^0 \oplus V^0 \& V^0 \\
\vdots \& \vdots \& \vdots \\
V^{n-1} \& V^{n-1} \oplus V^{n-1} \& V^{n-1} \\
E^n \& F^n \& E^n \\
}; 
\path[->] 
(m-2-1) edge (m-3-1) (m-3-1) edge (m-4-1)
(m-2-2) edge (m-3-2) (m-2-3) edge (m-3-3)
(m-3-2) edge (m-4-2) (m-3-3) edge (m-4-3)
; 
\path[>->]
(m-1-1) edge (m-1-2) edge (m-2-1)
(m-1-2) edge (m-2-2)
(m-1-3) edge (m-2-3)
(m-2-1) edge (m-2-2)
(m-4-1) edge (m-4-2)
(m-5-1) edge (m-5-2)
;
\path[->>]
(m-4-1) edge (m-5-1) (m-4-2) edge (m-5-2) (m-4-3) edge (m-5-3)
(m-5-2) edge (m-5-3)
(m-1-2) edge (m-1-3)
(m-2-2) edge (m-2-3)
(m-4-2) edge (m-4-3)
;
\end{tikzpicture} 
}
\]
where $V^k \in \nu$ for every $0 \leq k \leq n-1$. Notice that $V^j \oplus V^j \in \nu \subseteq \mathcal{GI}_{(\mathcal{Z, W})}$ since $\nu$ is closed under finite coproducts. Thus, by the dual of \cite[Coroll. 4.10]{BMS}, ($\mathsf{rs1}$) and the fact that $ \Gid_{(\mathcal{Z, W})}(W) \leq n$, we get that $F^n \in \mathcal{GI}_{(\mathcal{Z, W})} \subseteq \mathcal{Z}^\perp$. At this point, we aim to show that $E^n \in \mathcal{GI}_{(\mathcal{Z,W})}$. By the dual of \cite[Thm. 3.32]{BMS}, the idea is to prove that $E^n \in \mathcal{Z}^\perp$ and to construct a $\nu$-resolution 
\begin{align}
\cdots & \to V_1 \to V_0 \twoheadrightarrow E^{n} \label{eqn:resnu}
\end{align} 
of $E^n$ with cycles in $\mathcal{Z}^{\perp}$. 

Firstly, from the long Ext cohomology sequence obtained from the bottom row of the previous diagram, we can note that 
\begin{align}
\Ext^{i}(Z, E^{n}) & \cong \Ext^{i+1}(Z, E^{n}) \label{eqn:natiso1}
\end{align} 
for every $i \geq 1$ and $Z \in \mathcal{Z}$, since $F^n \in \mathcal{Z}^\perp$. On the other hand, since $M \in \mathcal{I}_\mathcal{Z}^{<\infty}$, $\nu \subseteq \mathcal{I}_\mathcal{Z}^{<\infty}$, and the class $\mathcal{I}_\mathcal{Z}^{<\infty}$ is closed under monocokernels, we have that $E^n \in \mathcal{I}_\mathcal{Z}^{<\infty}$. By \eqref{eqn:natiso1}, the previous implies that $E^n \in \mathcal{Z}^\perp$. 

We now construct the resolution \eqref{eqn:resnu}. Since $F^n \in \mathcal{GI}_{(\mathcal{Z, W})}$, there is an exact sequence $F^{n-1} \rightarrowtail V_0 \twoheadrightarrow F^{n}$ with $V_0 \in \nu$ and $F^{n-1} \in \mathcal{GI}_{(\mathcal{Z, W})}$ (again, by the dual of \cite[Thm. 3.32]{BMS}). Now, consider the following pullback diagram
\[
\begin{tikzpicture}[description/.style={fill=white,inner sep=2pt}] 
\matrix (m) [matrix of math nodes, row sep=2.3em, column sep=2.3em, text height=1.25ex, text depth=0.25ex] 
{ 
{F^{n-1}} & {F^{n-1}} & {} \\
{E^{n-1}} & {V_0} & {E^n} \\
{E^{n}} & {F^{n}} & {E^n} \\
}; 
\path[->] 
(m-2-1)-- node[pos=0.5] {\footnotesize$\mbox{\bf pb}$} (m-3-2) 
;
\path[>->]
(m-1-1) edge (m-2-1) 
(m-1-2) edge (m-2-2)
(m-2-1) edge (m-2-2) 
(m-3-1) edge (m-3-2)
;
\path[->>]
(m-2-1) edge (m-3-1) 
(m-2-2) edge (m-3-2) edge (m-2-3)
(m-3-2) edge (m-3-3)
;
\path[-,font=\scriptsize]
(m-1-1) edge [double, thick, double distance=2pt] (m-1-2)
(m-2-3) edge [double, thick, double distance=2pt] (m-3-3)
;
\end{tikzpicture} .
\]
Note that $F^{n-1}, E^{n} \in \mathcal{Z}^{\perp}$ implies $E^{n-1}\in \mathcal{Z}^{\perp}$. Now, from the left-hand side column and the middle row, we get a solid diagram as follows
\[
\begin{tikzpicture}[description/.style={fill=white,inner sep=2pt}] 
\matrix (m) [matrix of math nodes, row sep=2.3em, column sep=2.3em, text height=1.25ex, text depth=0.25ex] 
{ 
& {E^{n-1}} & {E^{n-1}} \\
{F^{n-1}} & {Q^{n-1}} & {V_0} \\
{F^{n-1}} & {E^{n-1}} & {E^n} \\
}; 
\path[->] 
(m-2-2)-- node[pos=0.5] {\footnotesize$\mbox{\bf pb}$} (m-3-3) 
;
\path[>->]
(m-1-3) edge (m-2-3) 
(m-1-2) edge (m-2-2)
(m-2-1) edge (m-2-2) 
(m-3-1) edge (m-3-2)
;
\path[->>]
(m-2-3) edge (m-3-3) 
(m-2-2) edge (m-2-3) edge (m-3-2)
(m-3-2) edge (m-3-3)
;
\path[-,font=\scriptsize]
(m-1-2) edge [double, thick, double distance=2pt] (m-1-3)
(m-2-1) edge [double, thick, double distance=2pt] (m-3-1)
;
\end{tikzpicture} 
\]
where the middle row splits since $\Ext^{1}(\nu, \mathcal{GI}_{(\mathcal{Z, W})}) = 0$. It then follows that $Q^{n-1} \simeq F^{n-1} \oplus V_0 \in \mathcal{GI}_{(\mathcal{Z, W})}$. Thus, we have two exact sequences 
\begin{align*}
E^{n-1} & \rightarrowtail V_0 \twoheadrightarrow E^{n} & & \text{and} & E^{n-1} & \rightarrowtail Q^{n-1} \twoheadrightarrow E^{n-1}
\end{align*} 
with $V_0 \in \nu$, $E^{n-1} \in \mathcal{Z}^{\perp}$ and $Q^{n-1} \in \mathcal{GI}_{(\mathcal{Z, W})}$. Therefore, proceeding with a similar argument applied to these two sequences, and after repeating the process inductively, we obtain \eqref{eqn:resnu}.
\end{proof}

The following result is another version of \cite[Prop. 7.3]{HMP2023m}.

\begin{proposition}\label{prop: direct summd of pi1GP}
Let $\mcC$ be an AB4 abelian category. If $(\mathcal{A,B})$ is a GP-admissible pair such that $\omega$ is closed under coproducts, and if every object in $\mathcal{GP}_{(\mathcal{A,B})}$ admits a $\omega$-resolution, then
\begin{align*}
\mathcal{GP}_{(\mathcal{A,B})} & = \mathcal{GP}_{(\omega,\mcB)} = {\rm add}(\pi\mathcal{GP}_{(\omega,\mcB,1)}).
\end{align*}
\end{proposition}

\begin{proof}
From \cite[Thm. 3.32]{BMS} we have that the equality $\mathcal{WGP}_{(\omega, \mcB)}=\mathcal{GP}_{(\mcA, \mcB)}$. On the other hand, from the definition of the class $\mathcal{WGP}_{(\omega, \mcB)}$ and the assumption on the objects in $\mathcal{GP}_{(\mathcal{A,B})}$, one can find for every $M \in \mathcal{WGP}_{(\omega, \mcB)}$ an exact complex
\[
W_{\bullet} \colon \cdots \to W_1 \to W_0 \to W_{-1} \to \cdots
\]
with $W_k \in \omega$ and $Z_k(W_\bullet) \in {}^{\perp}\mcB$ for every $k \in \mathbb{Z}$, and such that $M \simeq Z_0(W_{\bullet})$. Note that $Z_k(W_\bullet) \in {}^{\perp}\mcB$ holds for every $k > 0$ since $\omega \subseteq \mathcal{WGP}_{(\omega, \mcB)} = \mathcal{GP}_{(\mcA, \mcB)}$, $\mathcal{GP}_{(\mcA, \mcB)}$ is left thick and $\mathcal{GP}_{(\mcA, \mcB)} \subseteq {}^{\perp}\mcB$ (see \cite[Corolls. 3.15 \& 3.33]{BMS}). The rest of the proof follows as in the ``only if'' part of \cite[Prop. 7.3]{HMP2023m}. 
\end{proof}

In the following results, we shall assume that every 1-periodic $(\omega,\mathcal{B})$-Gorenstein projective object has finite injective dimension relative to $\mathcal{Z}$, and that every 1-periodic $(\mathcal{Z},\nu)$-Gorenstein injective object has finite projective dimension relative to $\mathcal{B}$. So let us first provide a characterization for these assumptions.

\begin{remark}\label{rmk:equiv_assumptions}
In an AB4 abelian category $\mathcal{C}$ where $(\mathcal{A,B})$ is a GP-admissible pair with $\omega$ closed under coproducts and such that every object in $\mathcal{GP}_{(\mathcal{A,B})}$ has a $\omega$-resolution, we have that the following assertions are equivalent for any class $\mathcal{Z} \subseteq \mathcal{C}$:
\begin{enumerate}
\item[(a)] $\pi\mathcal{GP}_{(\omega,\mcB,1)} \subseteq \mathcal{I}^{< \infty}_{\mathcal{Z}}$. 

\item[(b)] $\mathcal{GP}_{(\mcA,\mcB)} \subseteq \mathcal{I}^{< \infty}_{\mathcal{Z}}$. 

\item[(c)] $\mathcal{GP}^\wedge_{(\mcA,\mcB)} \subseteq \mathcal{I}^{< \infty}_{\mathcal{Z}}$. 
\end{enumerate}
Indeed, if we assume (a) and consider $M \in \mathcal{GP}_{(\mcA,\mcB)}$, by Proposition \ref{prop: direct summd of pi1GP} we can find a finite family $(M_i)^{l}_{i = 1}$ of objects in $\pi\mathcal{GP}_{(\omega,\mcB,1)}$ such that $M$ is a direct summand of $\bigoplus^l_{i = 1} M_i$. Since $\mathcal{I}^{< \infty}_{\mathcal{Z}}$ is clearly closed under finite coproducts and direct summands, we have that $M \in \mathcal{I}^{< \infty}_{\mathcal{Z}}$. Now if we assume (b), the containment $\mathcal{GP}^\wedge_{(\mcA,\mcB)} \subseteq \mathcal{I}^{< \infty}_{\mathcal{Z}}$ follows from the fact that $\mathcal{I}^{< \infty}_{\mathcal{Z}}$ is closed under monocokernels. Finally, the implication (c) $\Rightarrow$ (a) is trivial since every 1-periodic $(\omega,\mathcal{B})$-Gorenstein projective is $(\mathcal{A,B})$-Gorenstein projective. 
\end{remark}

\begin{proposition}\label{pro: Gid(GP)<=n}
Let $\mcC$ be an AB4 and AB4${}^\ast$ abelian category. If conditions from ($\mathsf{rs1}$) to ($\mathsf{rs5}$) listed in the relative setting are satisfied by a GP-admissible pair $(\mathcal{A,B})$ and a GI-admissible pair $(\mathcal{Z,W})$, then the following assertions hold true:
\begin{enumerate}
\item If $\pi\mathcal{GP}_{(\omega,\mcB,1)} \subseteq \mathcal{I}^{< \infty}_{\mathcal{Z}}$, then $\Gid_{(\mathcal{Z, W})}(\mathcal{GP}_{(\mcA, \mcB)}) = \Gid_{(\mathcal{Z, W})}(\omega)$.
 
\item If $\pi\mathcal{GI}_{(\mathcal{Z},\nu,1)} \subseteq \mathcal{P}^{< \infty}_{\mathcal{B}}$, then $\Gpd_{(\mathcal{A,B})}(\mathcal{GI}_{(\mathcal{Z},\mathcal{W})}) = \Gpd_{(\mathcal{A,B})}(\nu)$.
\end{enumerate}
\end{proposition}

\begin{proof}
We only prove in (1) the inequality $\Gid_{(\mathcal{Z, W})}(\mathcal{GP}_{(\mcA, \mcB)}) \leq \Gid_{(\mathcal{Z, W})}(\omega)$ (the other one is trivial). Without loss of generality, we may let $n := \Gid_{(\mathcal{Z, W})}(\omega) < \infty$. Now given $M\in \mathcal{GP}_{(\mcA, \mcB)}$, by Proposition~\ref{prop: direct summd of pi1GP} there exists a finite family $(M_i)^{l}_{i = 1}$ of objects in $\pi\mathcal{GP}_{(\omega,\mcB,1)}$ and $\widetilde{M} \in \mathcal{C}$ such that $M \oplus \widetilde{M} \simeq \bigoplus^l_{i = 1} M_i$. Note also that $\widetilde{M} \in {\rm add}(\pi\mathcal{GP}_{(\omega, \mcB, 1)}) = \mathcal{GP}_{(\mcA,\mcB)}$. On the other hand, by Lemma~\ref{lem: Gid(piXcapY)<=n} we have $\Gid_{(\mathcal{Z, W})}(M_i) \leq n$ for every $1 \leq i \leq l$. Since the class $\mathcal{GI}_{(\mathcal{Z,W})}$ is closed under finite coproducts, and finite coproducts of exact sequences are exact, we deduce that $\Gid_{(\mathcal{Z, W})}(M \oplus \widetilde{M}) \leq n$. Now by the dual of \cite[Thm. 4.1 (a)]{BMS}, ($\mathsf{rs4}$) and a pushout argument, we can form two exact sequences
\begin{align*}
M & \rightarrowtail V^0 \to \cdots \to V^{n-1} \twoheadrightarrow E^n, \\
\widetilde{M} & \rightarrowtail \widetilde{V}^0 \to \cdots \to \widetilde{V}^{n-1} \twoheadrightarrow \widetilde{E}^n,
\end{align*}
with $V^j, \widetilde{V}^j \in \nu$ for every $0 \leq j \leq n-1$, from which we get the exact sequence 
\[
M \oplus \widetilde{M} \rightarrowtail V^0 \oplus \widetilde{V}^0 \to \cdots \to V^{n-1} \oplus \widetilde{V}^{n-1} \twoheadrightarrow E^n \oplus \widetilde{E}^n
\]
with each $V^j \oplus \widetilde{V}^j \in \nu \subseteq \mathcal{GI}_{(\mathcal{Z,W})}$. Since $\Gid_{(\mathcal{Z, W})}(M\oplus \widetilde{M})\leq n$, it follows that $E^n \oplus \widetilde{E}^n \in \mathcal{GI}_{(\mathcal{Z, W})}$, and so $E^n \in \mathcal{GI}_{(\mathcal{Z, W})}$ by \cite[Coroll. 3.33]{BMS}. Therefore, $\Gid_{(\mathcal{Z,W})}(M) \leq n$. 
\end{proof}

\begin{remark}\label{rmk: GidGP^<=n}
From \cite[Coroll. 4.10]{BMS} we know that $(\mathcal{GP}_{(\mathcal{A,B})}, \omega)$ is a left Frobenius pair for any GP-admissible pair $(\mathcal{A,B})$ with $\omega$ closed under direct summands. Then, it follows by \cite[Thm. 2.11]{BMPS} that $\mathcal{GP}_{(\mathcal{A,B})}^\wedge$ is thick. Dually, one has that the class $\mathcal{GI}_{(\mathcal{Z,W})}^\vee$ is also thick for any GI-admissible pair $(\mathcal{Z,W})$ with $\nu$ closed under direct summands.
\end{remark}

The following result extends Proposition~\ref{pro: Gid(GP)<=n} and describes the relation between the classes $\mathcal{GP}^\wedge_{(\mcA,\mcB)}$ and $\mathcal{GI}^\vee_{(\mathcal{Z,W})}$.

\begin{proposition}\label{pro: GidGP^<=n}
Let $\mcC$ be an AB4 and AB4${}^\ast$ abelian category. If conditions from ($\mathsf{rs1}$) to ($\mathsf{rs5}$)  are satisfied by a GP-admissible pair $(\mathcal{A,B})$ and a GI-admissible pair $(\mathcal{Z,W})$, then the following assertions hold true: 
\begin{enumerate}
\item If $\pi\mathcal{GP}_{(\omega,\mcB,1)} \subseteq \mathcal{I}^{< \infty}_{\mathcal{Z}}$, then $\Gid_{(\mathcal{Z, W})}(\mathcal{GP}_{(\mcA,\mcB)}^{\wedge}) = \Gid_{(\mathcal{Z,W})}(\omega)$.
\item If $\pi\mathcal{GI}_{(\mathcal{Z},\nu,1)} \subseteq \mathcal{P}^{< \infty}_{\mathcal{B}}$, then $\Gpd_{(\mathcal{A,B})}(\mathcal{GI}_{(\mathcal{Z},\mathcal{W})}^{\vee}) = \Gpd_{(\mathcal{A,B})}(\nu)$.
\end{enumerate}
\end{proposition}

\begin{proof}
We only prove $\Gid_{(\mathcal{Z, W})}(\mathcal{GP}_{(\mcA,\mcB)}^{\wedge}) \leq \Gid_{(\mathcal{Z,W})}(\omega)$ in part (1). Again, without loss of generality, we may let $n := \Gid_{(\mathcal{Z, W})}(\omega) < \infty$. Let $M \in \mathcal{GP}_{(\mcA,\mcB)}^{\wedge}$ with $\Gpd_{(\mathcal{A,B})}(M) = k < \infty$. We use induction on $k$. If $k = 0$ the result follows by Proposition~\ref{pro: Gid(GP)<=n}. For $k \geq 1$, we take a short exact sequence $M' \rightarrowtail G \twoheadrightarrow M$ with $G \in \mathcal{GP}_{(\mathcal{A,B})}$ and $\Gpd_{(\mathcal{A,B})}(M') = k - 1$. Using the dual of \cite[Coroll. 4.11]{BMS} and induction hypothesis, we get that $\id_{\nu}(G) = \Gid_{(\mathcal{Z,W})}(G) \leq n$ and $\id_{\nu}(M') = \Gid_{(\mathcal{Z,W})}(M') \leq n$. Finally, the previous along with Remark \ref{rmk: GidGP^<=n}, imply that $\Gid_{(\mathcal{Z, W})}(M) = \id_{\nu}(M) \leq n$. 
\end{proof}

%%%%%%%%%%%%%%%%%%%%%%%%%%%%%%%%%%%%%
%%%%%%%%%%%%%%%%%%%%%%%%%%%%%%%%%%%%%
%%%%%%%%%%%%%%%%%%%%%%%%%%%%%%%%%%%%%
%%%%%%%%%%%%%%%%%%%%%%%%%%%%%%%%%%%%%
%%%%%%%%%%%%%%%%%%%%%%%%%%%%%%%%%%%%%
%%%%%%%%%%%%%%%%%%%%%%%%%%%%%%%%%%%%%

\section{Applications to relative finitistic and \\ global Gorenstein dimensions}\label{sec:dimensions}

As mentioned in the introduction of \cite{HMP2023m}, one of the main applications of \cite[Thm. 2.7]{BM2} by Bennis and Mahdou was to show in \cite[Thm. 1.1]{BMglobal} the equality
\begin{align}
\sup\{\Gpd_{R}(M) : M\in \Mod(R) \} & = \sup\{\Gid_{R}(M) : M\in \Mod(R)\}. \label{eqn:BennisMahdou_global}
\end{align}
for any associative ring $R$ with identity. In other words, the (left) global Gorenstein projective and Gorenstein injective dimensions of any ring coincide, and their common value, known as the \emph{global Gorenstein dimension}, will be denoted by ${\rm gl.GD}(R)$. 

On the other hand, global Gorenstein dimensions relative to GP-admissible pairs were studied by Becerril in \cite{Becerril}. So a natural question is whether it is possible to extend equality \eqref{eqn:BennisMahdou_global} for global Gorenstein projective and Gorenstein injective dimensions relative to GP-admissible and GI-admissible pairs. We partially answer this in the positive, under certain conditions for such pairs. Our approach will consider relative finitistic Gorenstein dimensions. 

We can use Proposition \ref{pro: GidGP^<=n} to show that the finitistic relative Gorenstein projective and injective dimensions of $\mcC$ are equal under the relative setting. These dimensions were defined in \cite[Def. 4.16]{BMS} as
\begin{align*}
{\rm FGPD}_{(\mathcal{A,B})}(\mcC) & := {\rm resdim}_{\mathcal{GP}_{(\mcA,\mcB)}}(\mathcal{GP}^\wedge_{(\mcA,\mcB)}), \\
{\rm FGID}_{(\mathcal{Z,W})}(\mcC) & := {\rm coresdim}_{\mathcal{GI}_{(\mathcal{Z,W})}}(\mathcal{GI}^\vee_{(\mathcal{Z,W})}).
\end{align*}

\begin{theorem}\label{thm: gpd leq gid}
Let $\mcC$ be an AB4 and AB4${}^\ast$ abelian category. If conditions from ($\mathsf{rs1}$) to ($\mathsf{rs5}$)  are satisfied by a GP-admissible pair $(\mathcal{A,B})$ and a GI-admissible pair $(\mathcal{Z,W})$, such that $\mathcal{GP}_{(\mathcal{A,B})}^{\wedge} = \mathcal{GI}_{(\mathcal{Z,W})}^{\vee}$ and with $\mcB$ and $\mathcal{Z}$ closed under direct summands, then 
\begin{align*}
\rm{FGID}_{(\mathcal{Z,W})}(\mcC) & = \Gid_{(\mathcal{Z,W})}(\mathcal{GP}_{(\mathcal{A,B})}^\wedge) = \Gid_{(\mathcal{Z,W})}(\mathcal{GP}_{(\mathcal{A,B})}) = \Gid_{(\mathcal{Z,W})}(\omega) \\
& = \id_{\nu}(\omega) = \id_{\nu}(\mathcal{GP}_{(\mathcal{A,B})}) = \id_{\nu}(\mathcal{GP}_{(\mathcal{A,B})}^\wedge) = \id_{\mathcal{Z}}(\mathcal{GP}_{(\mathcal{A,B})}^\wedge) \\
& = \id_{\mathcal{Z}}(\mathcal{GP}_{(\mathcal{A,B})}) = \pd_{\mathcal{B}}(\mathcal{GI}_{(\mathcal{Z,W})}) = \pd_{\mathcal{B}}(\mathcal{GI}_{(\mathcal{Z,W})}^\vee) = \pd_{\omega}(\mathcal{GI}_{(\mathcal{Z,W})}^\vee) \\
& = \pd_{\omega}(\mathcal{GI}_{(\mathcal{Z,W})}) = \pd_{\omega}(\nu) = \Gpd_{(\mathcal{A,B})}(\nu) = \Gpd_{(\mathcal{A,B})}(\mathcal{GI}_{(\mathcal{Z,W})}) \\
& = \Gpd_{(\mathcal{A,B})}(\mathcal{GI}_{(\mathcal{Z,W})}^\vee) = \rm{FGPD}_{(\mathcal{A,B})}(\mcC).
\end{align*}
\end{theorem}

\begin{proof}
By \cite[Lem. 2.6, Corolls. 4.11 (a) \& 4.15 (a1)]{BMS} we have that 
\begin{align*}
\Gpd_{(\mathcal{A,B})}(\mathcal{GI}_{(\mathcal{Z,W})}) & = \pd_{\omega}(\mathcal{GI}_{(\mathcal{Z,W})}) = \pd_{\omega}(\mathcal{GI}_{(\mathcal{Z,W})}^\vee) = \pd_{\omega}(\mathcal{GP}_{(\mathcal{A,B})}^\wedge) \\
& = \rm{FGPD}_{(\mathcal{A,B})}(\mcC) = \pd_{\mathcal{B}}(\mathcal{GP}_{(\mathcal{A,B})}^\wedge) = \pd_{\mathcal{B}}(\mathcal{GI}_{(\mathcal{Z,W})}^\vee) \\
& = \pd_{\mathcal{B}}(\mathcal{GI}_{(\mathcal{Z,W})}),
\end{align*}
and
\[
\Gpd_{(\mathcal{A,B})}(\nu) = \pd_{\omega}(\nu).
\]
Dually, 
\begin{align*}
\Gid_{(\mathcal{Z,W})}(\mathcal{GP}_{(\mathcal{A,B})}) & = \id_{\nu}(\mathcal{GP}_{(\mathcal{A,B})}) = \id_{\nu}(\mathcal{GP}_{(\mathcal{A,B})}^\wedge) = \id_{\nu}(\mathcal{GI}_{(\mathcal{Z,W})}^\vee) \\
& = \rm{FGID}_{(\mathcal{Z,W})}(\mcC) = \id_{\mathcal{Z}}(\mathcal{GI}_{(\mathcal{Z,W})}^\vee) = \id_{\mathcal{Z}}(\mathcal{GP}_{(\mathcal{A,B})}^\wedge) \\
& = \id_{\mathcal{Z}}(\mathcal{GP}_{(\mathcal{A,B})}),
\end{align*}
and
\[
\Gid_{(\mathcal{Z,W})}(\omega) = \id_{\nu}(\omega).
\]
Note also that the equality $\mathcal{GP}_{(\mathcal{A,B})}^{\wedge} = \mathcal{GI}_{(\mathcal{Z,W})}^{\vee}$, along with \cite[Coroll. 4.15 (a)]{BMS} and its dual, imply the validity of the containments $\mathcal{GP}_{(\mcA,\mcB)} \subseteq \mathcal{I}^{< \infty}_{\mathcal{Z}}$ and $\mathcal{GI}_{(\mathcal{Z,W})} \subseteq \mathcal{P}^{< \infty}_{\mathcal{B}}$. Then by Propositions \ref{pro: Gid(GP)<=n} and \ref{pro: GidGP^<=n}, we have that 
\begin{align*}
\Gid_{(\mathcal{Z,W})}(\mathcal{GP}_{(\mcA,\mcB)}^{\wedge}) & = \Gid_{(\mathcal{Z, W})}(\mathcal{GP}_{(\mcA, \mcB)}) = \Gid_{(\mathcal{Z, W})}(\omega), \\ 
\Gpd_{(\mathcal{A,B})}(\mathcal{GI}_{(\mathcal{Z,W})}^{\vee}) & = \Gpd_{(\mathcal{A,B})}(\mathcal{GI}_{(\mathcal{Z,W})}) = \Gpd_{(\mathcal{A,B})}(\nu).
\end{align*}
Therefore, since $\pd_{\omega}(\nu) = \id_{\nu}(\omega)$, the result follows.
\end{proof}

Related to the finitistic relative Gorenstein dimensions, we also have the global $(\mcA,\mcB)$-Gorenstein projective and global $(\mathcal{Z,W})$-Gorenstein injective dimensions of $\mathcal{C}$, also defined in \cite[Def. 4.17]{BMS}, as
\begin{align*}
{\rm gl.GPD}_{(\mathcal{A,B})}(\mcC) & := {\rm Gpd}_{(\mcA,\mcB)}(\mathcal{C}), \\
{\rm gl.GID}_{(\mathcal{Z,W})}(\mcC) & := {\rm Gid}_{(\mathcal{Z,W})}(\mathcal{C}).
\end{align*}
One can note that 
\begin{align*}
{\rm FGPD}_{(\mathcal{A,B})}(\mcC) & \leq {\rm gl.GPD}_{(\mathcal{A,B})}(\mcC), \\
\rm{FGID}_{(\mathcal{Z,W})}(\mcC) & \leq {\rm gl.GID}_{(\mathcal{Z,W})}(\mcC).
\end{align*}
Note that $\mcC = \mathcal{GP}^\wedge_{(\mcA,\mcB)}$ implies that ${\rm FGPD}_{(\mathcal{A,B})}(\mcC) = {\rm gl.GPD}_{(\mathcal{A,B})}(\mcC)$, and a similar equality holds for the injective case if $\mathcal{C} = \mathcal{GI}^\vee_{(\mathcal{Z,W})}$. It follows that under the assumptions of Theorem \ref{thm: gpd leq gid}, one obtains the equality
\begin{align}
{\rm gl.GID}_{(\mathcal{Z,W})}(\mcC) & = \rm{FGID}_{(\mathcal{Z,W})}(\mcC) = {\rm FGPD}_{(\mathcal{A,B})}(\mcC) = {\rm gl.GPD}_{(\mathcal{A,B})}(\mcC). \label{eqn:BennisMahdou_relative}
\end{align}
In particular, Theorem \ref{thm: gpd leq gid} covers \cite[Coroll. 3.5]{Becerril} in the case where $\mcC$ has enough projective and injective objects, and $\omega$ and $\nu$ coincide with the class of projective and injective objects of $\mcC$, respectively. The assumption $\mcC = \mathcal{GP}^\wedge_{(\mcA,\mcB)} = \mathcal{GI}^\vee_{(\mathcal{Z,W})}$ is sufficient but not necessary. For instance, in the absolute case where $\mcA = \mcB =$ projective $R$-modules and $\mathcal{Z} = \mathcal{W} =$ injective $R$-modules, the equality \eqref{eqn:BennisMahdou_global} holds for any arbitrary ring $R$. In the relative setting, there are some known cases where \eqref{eqn:BennisMahdou_relative} holds (see \cite[Prop. 3.15]{Becerril}). The following extends the cited result.

\begin{corollary}\label{coro:relative_global}
Let $\mcC$ be an AB4 and AB4${}^\ast$ abelian category, and $(\mathcal{A,B})$ be GP-admissible pair and $(\mathcal{Z,W})$ be a GI-admissible pair satisfying conditions from ($\mathsf{rs1}$) to ($\mathsf{rs5}$), with $\mathcal{B}$ and $\mathcal{Z}$ closed under direct summands. If ${\rm gl.GPD}_{(\mathcal{A,B})}(\mcC) < \infty$ and $\omega \subseteq \mcGI^\vee_{(\mathcal{Z,W})}$, then 
\begin{align*}
\Gid_{(\mathcal{Z,W})}(\omega) & = \id_{\nu}(\omega) = \pd_{\omega}(\nu) = \Gpd_{(\mathcal{A,B})}(\nu).
\end{align*}
Moreover, if $\mathcal{GP}_{(\mathcal{A,B})} \subseteq \mathcal{I}^{<\infty}_{\mathcal{Z}}$ then the previous equality extends to
\begin{align*}
{\rm gl.GID}_{(\mathcal{Z,W})}(\mcC) & = \Gid_{(\mathcal{Z,W})}(\omega) = \pd_{\omega}(\nu) = \pd(\nu) = \pd(\mathcal{Z}) \\
& = \id(\mathcal{B}) = \id(\omega) = \id_{\nu}(\omega) = \Gpd_{(\mathcal{A,B})}(\nu) = {\rm gl.GPD}_{(\mathcal{A,B})}(\mcC).
\end{align*}
\end{corollary}

\begin{proof}
The first equality follows by \cite[Coroll. 4.15]{BMS} and its dual. The last equality will be a consequence of Theorem \ref{thm: gpd leq gid} after showing that $\mathcal{C} = \mathcal{GI}_{(\mathcal{Z,W})}^{\vee}$ (note that we already have $\mathcal{GP}_{(\mathcal{A,B})}^{\wedge} = \mathcal{C}$). Since $\mathcal{C} \subseteq \mathcal{I}^{<\infty}_{\mathcal{Z}}$ by Remark \ref{rmk:equiv_assumptions}, we obtain from Proposition \ref{pro: GidGP^<=n} that $\Gid_{(\mathcal{Z, W})}(\mathcal{C}) = \Gid_{(\mathcal{Z,W})}(\omega) < \infty$. 
\end{proof}

Assumptions in Corollary \ref{coro:relative_global} are sufficient but not necessary conditions so that ${\rm gl.GID}_{(\mathcal{Z,W})}(\mcC) = {\rm gl.GPD}_{(\mathcal{A,B})}(\mcC)$. As we already mentioned, the previous equality always holds for the choice $\mcA = \mcB = \mathcal{P}(R)$ (projective $R$-modules) and $\mathcal{Z} = \mathcal{W} = \mathcal{I}(R)$ (injective $R$-modules), even in the case where the global Gorenstein dimension is infinite. Let us now analyze some known results and outcomes with respect to other very common GP-admissible and GI-admissible pairs in relative Gorenstein homological algebra, namely, 
\begin{align}
(\mcA,\mcB) & = (\mcP(R),\mcF(R)) & & \text{and} & (\mathcal{Z,W}) & = (\text{FP-}\mcI(R),\mcI(R)) \label{eqn:pairsDing}
\end{align}
and 
\begin{align}
(\mcA,\mcB) & = (\mcP(R),\mcL(R)) & & \text{and} & (\mathcal{Z,W}) & = (\text{FP}_\infty\text{-}\mcI(R),\mcI(R)). \label{eqn:pairsAC}
\end{align}

\begin{itemize}
\item Regarding the pairs \eqref{eqn:pairsDing}, we obtain the classes $\mathcal{DP}(R) := \mathcal{GP}_{(\mcP(R),\mcF(R))}$ and $\mathcal{DI}(R) := \mathcal{GI}_{(\text{FP-}\mcI(R),\mcI(R))}$ of Ding projective and Ding injective $R$-modules, respectively. Let us denote the corresponding global dimensions, called the (\emph{left}) \emph{global Ding projective} and the \emph{global Ding injective dimensions} of $R$, by ${\rm gl.DPD}(R)$ and ${\rm gl.DID}(R)$. It is not known whether or not the equality ${\rm gl.DPD}(R) = {\rm gl.DID}(R)$ holds for any ring $R$, although there are some cases in which one can give an affirmative answer, like for instance when $R$ is a Ding-Chen ring (see Yang's \cite[Thm. 3.11]{Yang}).

Concerning the relation of the global Ding dimensions with the global Gorenstein dimension, it was shown by  Mahdou and Tamekkante in \cite[Thm. 3.1]{MahdouTamekkante} that ${\rm gl.DPD}(R) = {\rm gl.GD}(R)$ for any ring $R$. It then follows that 
\begin{align}
{\rm gl.GD}(R) & = {\rm gl.DPD}(R) = {\rm gl.DID}(R). \label{eqn:globalGorensteinVsDing}
\end{align}
whenever $R$ is a Ding-Chen ring. The equality \eqref{eqn:globalGorensteinVsDing} was also proved in \cite[Prop. 3.16]{Becerril} for rings which satisfy that $\pd(\text{FP}\text{-}\mcI(R)) < \infty$ and $\id(\mathcal{P}(R)) < \infty$. Another family of rings for which \eqref{eqn:globalGorensteinVsDing} is also valid is formed by those rings $R$ such that both $\pd(\mathcal{F}(R))$ and $\id(\text{FP}\text{-}\mcI(R))$ are smaller than or equal to ${\rm gl.GD}(R)$. Indeed, one can mimic the proof of Liang and Wang's \cite[Lems. 13 \& 15]{LiangWang} in order to show that
\begin{align*}
{\rm gl.DPD}(R) & = {\rm max}\{ {\rm gl.GD}(R), \pd(\mathcal{F}(R)) \}, \\
{\rm gl.DID}(R) & = {\rm max}\{ {\rm gl.GD}(R), \id(\text{FP}\text{-}\mcI(R)) \}.
\end{align*}
Thus, $\pd(\mathcal{F}(R)), \id(\text{FP}\text{-}\mcI(R) \leq {\rm gl.GD}(R)$ imply \eqref{eqn:globalGorensteinVsDing}. 

We can deduce one more case in which \eqref{eqn:globalGorensteinVsDing} holds. Let $R$ be a ring such that ${\rm gl.GD}(R)$ is finite and over which every Ding projective $R$-module has finite injective dimension relative to $\text{FP}\text{-}\mcI(R)$. In this case, ${\rm gl.DPD}(R) = {\rm gl.GD}(R)$ is finite. Moreover, using that the finiteness of ${\rm gl.GD}(R)$ implies that of $\id(\mcP(R))$ (see \cite[Lem. 2.1]{BMglobal}) note that $\mathcal{P}(R) \subseteq \mathcal{DI}(R)^\vee$. Hence, from Corollary \ref{coro:relative_global} we have that ${\rm gl.DID}(R)$ is also finite and
\[
{\rm gl.DPD}(R) = {\rm gl.GD}(R) = {\rm gl.DID}(R) = \pd(\mcI(R)) = \id(\mcP(R)). 
\]

\item For the GP-admissible and GI-admissible pairs in \eqref{eqn:pairsAC}, the corresponding relative Gorenstein projective and Gorenstein injective modules, introduced in \cite[\S \ 5 \& 8]{BGH}, are known as Gorenstein AC-projective and Gorenstein AC-injective modules. The classes of these modules will be denoted by $\mcGP_{\rm AC}(R)$ and $\mcGI_{\rm AC}(R)$, and their corresponding global Gorenstein AC-projective and Gorenstein AC-injective dimensions by: 
\begin{align*}
\rm{gl.GPD}_{\rm{AC}}(R) & := \Gpd_{(\mcP(R),\mathcal{L}(R))}(\Mod(R)), \\
\rm{gl.GID}_{\rm{AC}}(R) & := \Gid_{(\text{FP}_\infty\text{-}\mcI(R),\mcI(R))}(\Mod(R)).
\end{align*}
By Corollary \ref{coro:relative_global}, if $R$ is a ring with $\rm{gl.GPD}_{\rm{AC}}(R) < \infty$ and such that every Gorenstein AC-projective module has finite injective dimension relative to $\text{FP}_\infty\text{-}\mcI(R)$, then $\rm{gl.GID}_{\rm{AC}}(R)$ is also finite and
\[
\rm{gl.GID}_{\rm{AC}}(R) = \rm{gl.GPD}_{\rm{AC}}(R).
\]
On the other hand, by \cite[Lems. 13 \& 15]{LiangWang} we know
\begin{align*}
{\rm gl.GPD}_{\rm AC}(R) & = {\rm max}\{ {\rm gl.GD}(R), \pd(\mathcal{L}(R)) \}, \\
{\rm gl.GID}_{\rm AC}(R) & = {\rm max}\{ {\rm gl.GD}(R), \id(\text{FP}_\infty\text{-}\mcI(R)) \}.
\end{align*}
It then follows that $\pd(\mathcal{L}(R))$ and $\id(\text{FP}_\infty\text{-}\mcI(R))$ are both finite, and so $\mcGP(R) = \mathcal{DP}(R) = \mcGP_{\rm AC}(R)$ and $\mcGI(R) = \mathcal{DI}(R) = \mcGI_{\rm AC}(R)$. Hence,
\begin{align}
{\rm gl.GD}(R) & = \rm{gl.DPD}(R) = {\rm gl.GPD}_{\rm{AC}}(R) = {\rm gl.GID}_{\rm{AC}}(R) = \rm{gl.DID}(R). \label{eqn:GorDingACwholeeq}
\end{align}
\end{itemize}

\begin{example}
There are cases in which \eqref{eqn:GorDingACwholeeq} holds and the common value is $\infty$. Let $\mathbb{K}$ be a field, and consider the $\mathbb{K}$-algebra of polynomials with coefficients in $\mathbb{K}$ and infinitely many variables $\mathbb{K}[x_1, x_2, \dots]$ along with the ideal $\mathfrak{m} = (x_1, x_2, \dots)$ (infinitely) generated by these variables. Let $R$ be the quotient commutative $\mathbb{K}$-algebra
\[
R = \mathbb{K}[x_1, x_2, \dots] / \mathfrak{m}^2.
\]
Note that $R$ can also be obtained as the path algebra $\mathbb{K}Q / I$ where $Q$ is the quiver
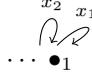
\[
\begin{tikzpicture}
  \node {$\cdots$ $\bullet_{1}$} edge [in=30,out=45,loop] node[above] {\scriptsize$x_1$} () edge [in=60,out=90,loop] node[above] {\scriptsize$x_2$} ();
\end{tikzpicture}
\]
with one vertex $1$, infinitely many loops $x_i$, and $I$ is the ideal generated by the terms $x_i x_j$ with $i$ and $j$ running over the positive integers. 

From \cite[Prop. 2.5]{BGH} we know that over this ring, the class of finitely generated and free $R$-modules coincides with $\mathcal{FP}_\infty(R)$. It then follows that $\text{FP}_\infty\text{-}\mcI(R) = \mathcal{L}(R) = \Mod(R)$. Thus, from \cite[Lems. 13 \& 15]{LiangWang} we obtain
\begin{align*}
{\rm gl.GPD}_{\rm AC}(R) & = {\rm max}\{ {\rm gl.GD}(R), \pd(\Mod(R)) \} = \pd(\Mod(R)) = {\rm gl.dim}(R), \\
{\rm gl.GID}_{\rm AC}(R) & = {\rm max}\{ {\rm gl.GD}(R), \id(\Mod(R)) \} = \id(\Mod(R)) = {\rm gl.dim}(R),
\end{align*}
where ${\rm gl.dim}(R)$ denotes the global dimension of $R$. On the other hand, the projective $R$-module $P(1)$ at $1$ has infinite injective dimension. Similarly, the injective $R$-module $I(1)$ at $1$ has infinite projective dimension. It then follows that 
\[
{\rm gl.GPD}_{\rm AC}(R) = {\rm gl.GID}_{\rm AC}(R) = {\rm gl.dim}(R) = \infty
\] 
and that $\pd(\mathcal{I}(R)) = \infty$ and $\id(\mcP(R)) = \infty$. In particular, $R$ is not a Gorenstein algebra (see \cite[Thm. 2.2 \& Def. 2.5]{BR}), and so ${\rm gl.GD}(R) =\infty$. This in turn implies that 
\[
\rm{gl.DPD}(R) = \rm{gl.DID}(R) = \infty.
\] 
Hence, for this $R$ the equality \eqref{eqn:GorDingACwholeeq} holds without the assumption ${\rm gl.GPD}_{\rm AC}(R) < \infty$. Moreover, over $R$ we can also note that $\mathcal{GP}_{\rm AC}(R)= \mathcal{P}(R)$ and that $\mathcal{GI}_{\rm AC}(R)= \mathcal{I}(R)$. Indeed, if $G$ is a Gorenstein AC-projective $R$-module, it is a cycle of a $\Hom(-,\Mod(R))$-acyclic exact complex $P_\bullet$ of projective $R$-modules. It follows that each exact sequence $Z_m(P_\bullet) \rightarrowtail P_m \twoheadrightarrow Z_{m-1}(P_\bullet)$ is split exact, and then $G$ is projective. So the other assumption in Corollary \ref{coro:relative_global}, namely $\mathcal{GP}_{\rm AC}(R) \subseteq \mathcal{I}_{\text{FP}_\infty\text{-}\mcI(R)}^{< \infty}$ becomes $\mathcal{P}(R) \subseteq \mathcal{I}^{< \infty}$, which does not hold for this ring $R$. Therefore, the hypotheses in Corollary \ref{coro:relative_global} are sufficient but not necessary conditions to obtain the equality ${\rm gl.GPD}_{(\mathcal{A,B})}(\mcC) = {\rm gl.GID}_{(\mathcal{Z,W})}(\mcC)$. 
\end{example}

%%%%%%%%%%%%%%%%%%%%%%%%%%%%%%%%%%%%%
%%%%%%%%%%%%%%%%%%%%%%%%%%%%%%%%%%%%%
%%%%%%%%%%%%%%%%%%%%%%%%%%%%%%%%%%%%%
%%%%%%%%%%%%%%%%%%%%%%%%%%%%%%%%%%%%%

\section*{Funding}

The authors thank Project PAPIIT-Universidad Nacional Aut\'onoma de M\'exico IN100124. The first author is supported by a postdoctoral fellowship from Direcci\'on General de Asuntos del Personal Acad\'emico DGAPA (Universidad Nacional Aut\'onoma de M\'exico). The third author is supported by the following institutions: Agencia Nacional de Investigaci\'on e Innovaci\'on (ANII) and Programa de Desarrollo de las Ciencias B\'asicas (PEDECIBA). 

%%%%%%%%%%%%%%%%%%%%%%%%%%%%%%%%%%%%%
%%%%%%%%%%%%%%%%%%%%%%%%%%%%%%%%%%%%%
%%%%%%%%%%%%%%%%%%%%%%%%%%%%%%%%%%%%%
%%%%%%%%%%%%%%%%%%%%%%%%%%%%%%%%%%%%%

\bibliographystyle{alpha}
\bibliography{biblionSG}

\end{document}